\documentclass[a4paper,12pt]{article}
\usepackage{amsmath}
\usepackage{amssymb}
\usepackage{amsthm}
\usepackage[english]{babel}

\newtheorem{theor}{Theorem}[section]
\newtheorem{defi}[theor]{Definition}
\newtheorem{lemma}[theor]{Lemma}

\newcommand{\Sign}{\mathrm{Sign}}

\begin{document}

\begin{center}
{\Large 
Sign conjugacy classes of the symmetric groups}

Lucia Morotti
\end{center}

\begin{abstract}
A conjugacy class $C$ of a finite group $G$ is a sign conjugacy class if every irreducible character of $G$ takes value 0, 1 or -1 on $C$. In this paper we classify the sign conjugacy classes of the symmetric groups and thereby verify a conjecture of Olsson.

\end{abstract}

\section{Introduction}

We will begin this paper by giving the definition of sign conjugacy class for an arbitrary finite group.

\begin{defi}
Let $G$ be a finite group. A conjugacy class of $G$ is a sign conjugacy class of $G$ if every irreducible character of $G$ takes values 0, 1 or -1 on $C$.
\end{defi}

Since we will be working with the symmetric group, we will consider partitions instead of conjugacy classes. A partition of $n$ is a sign partition if it is the corresponding conjugacy class of $S_n$ is a sign conjugacy class. An easy example of a sign partition of $n$ is $(n)$.

\begin{defi}
Define $\Sign$ to be the subsets of partitions consisting of all partitions $(\gamma_1,\ldots,\gamma_r)$ for which there exists an $s$, $0\leq s\leq r$, such that the following hold:
\begin{itemize}
\item
$\gamma_i>\gamma_{i+1}+\ldots+\gamma_r$ for $1\leq i\leq s$,

\item
$(\gamma_{s+1},\ldots,\gamma_r)$ is one of the following partitions:
\begin{itemize}
\item
$()$, $(1,1)$, $(3,2,1,1)$ or $(5,3,2,1)$,

\item
$(a,a-1,1)$ with $a\geq 2$,

\item
$(a,a-1,2,1)$ with $a\geq 4$,

\item
$(a,a-1,3,1)$ with $a\geq 5$.
\end{itemize}
\end{itemize}
\end{defi}

The name $\Sign$ for the above set is justified by the next theorem, which classifies sign partitions.

\begin{theor}\label{t1}A partition $\gamma$ is a sign partition if and only if $\gamma\in\Sign$.
\end{theor}

This was first formulated by Olsson in \cite{b7} as a conjecture.

In order to prove Theorem \ref{t1} we will use two results from \cite{b7}. The first one of them is the following lemma (Theorem 7 of \cite{b7}).

\begin{lemma}\label{l1}
A sign partition cannot have repeated parts, except possibly for the part 1, which may have multiplicity 2.
\end{lemma}

In particular only partitions of the form $(\gamma_1,\ldots,\gamma_r)$ with either $\gamma_1>\ldots>\gamma_r$ or $\gamma_1>\ldots>\gamma_{r-2}>\gamma_{r-1}=\gamma_r=1$ may be sign partitions. The next lemma can also be found in \cite{b7} (Proposition 2).

\begin{lemma}\label{l2}
Let $(\gamma_1,\ldots,\gamma_r)$ be a partition of $n$ and let $m>n$. Then $(\gamma_1,\ldots,\gamma_r)$ is a sign partition if and only if $(m,\gamma_1,\ldots,\gamma_r)$ is a sign partition.
\end{lemma}

For any partition $\lambda=(\lambda_1,\ldots,\lambda_k)$ let $|\lambda|:=\lambda_1+\ldots+\lambda_k$. Also for $1\leq i\leq k$ and $1\leq j\leq \lambda_i$ let $h_{i,j}^\lambda$ denote the hook length of the node $(i,j)$ of $\lambda$. For partitions $\lambda,\mu$ with $|\lambda|=n=|\mu|$ let $\chi^\lambda_\mu$ denote the value of the irreducible character of $S_n$ labeled by $\lambda$ on the conjugacy class with cycle partition $\mu$.

Together with the previous lemmas, the following theorem, which will be proved in Sections \ref{s1} and \ref{s2}, will allow us to prove one direction of Theorem \ref{t1}.

\begin{theor}\label{t2}
Let $\alpha=(\alpha_1,\ldots,\alpha_h)$ be a partition with $h\geq 3$. Assume that $\alpha_1>\alpha_2$, that $\alpha\not\in\Sign$ and that $(\alpha_2,\ldots,\alpha_h)\in\Sign$. Then if $\alpha\not=(5,4,3,2,1)$ we can find a partition $\beta$ of $|\alpha|$ such that $\chi^\beta_\alpha\not\in\{0,\pm 1\}$ and $h_{2,1}^\beta=\alpha_1$.
\end{theor}

The other direction of Theorem \ref{t1} will be proved using Lemma \ref{l2} and the results from Section \ref{s3}, where we prove that the partitions $(\gamma_{s+1},\ldots,\gamma_r)$ are sign partitions.

References about results on partitions and irreducible characters of $S_n$ can be found in \cite{b1} and \cite{b4}.

\section{Proof of Theorem \ref{t2} for $\alpha_2\leq\alpha_3+\ldots+\alpha_h$}\label{s1}

In this section we will prove Theorem \ref{t2} in the case where $\alpha_2\leq \alpha_3+\ldots+\alpha_h$. Since by assumption $h\geq 3$ and $(\alpha_2,\ldots,\alpha_h)\in\Sign$, we have that
\begin{eqnarray*}
(\alpha_2,\ldots,\alpha_h)\!\!&\!\!\in\!\!&\!\!\{(1,1),(3,2,1,1),(5,3,2,1)\}\!\cup\!\{(a,a-1,1):a\geq 2\}\\
&&\hspace{6pt}\cup\{(a,a-1,2,1):a\geq 4\}\!\cup\!\{(a,a-1,3,1):a\geq 5\}.
\end{eqnarray*}
Also $\alpha_1\leq \alpha_2+\ldots+\alpha_h$ as $\alpha\not\in\Sign$ and by assumption $\alpha_1>\alpha_2$. If
\begin{eqnarray*}
(\alpha_2,\ldots,\alpha_h)\!\!&\!\!\in\!\!&\!\!\{(1,1),(3,2,1,1),(5,3,2,1)\}\!\cup\!\{(a,a-1,1):2\leq a\leq 4\}\\
&&\hspace{6pt}\cup\{(a,a-1,2,1):4\leq a\leq 8\}\!\cup\!\{(a,a-1,3,1):5\leq a\leq 10\}
\end{eqnarray*}
there are only finitely many such $\alpha$ and it can be checked that for each one of them Theorem \ref{t2} holds.

For $(\alpha_2,\ldots,\alpha_h)=(a,a-1,1)$ with $a\geq 5$ let
\[\beta:=\left\{\begin{array}{ll}
(2a,2,1^{\alpha_1-2}),&a+2\leq\alpha_1\leq 2a-2\mbox{ or }\alpha_1=2a,\\
(a-1,a-1,a-1,4),&\alpha_1=a+1,\\
(2a,\alpha_1),&\alpha_1=2a-1.
\end{array}\right.\]

For $(\alpha_2,\ldots,\alpha_h)=(a,a-1,2,1)$ with $a\geq 9$ let
\[\beta:=\left\{\begin{array}{ll}
(2a+2,4,1^{\alpha_1-4}),&a+4\leq \alpha_1\leq 2a-2\mbox{ or }2a\leq\alpha_1\leq 2a+2,\\
(2a+2,\alpha_1-1,1),&\alpha_1=a+1,\\
(2a+2,2,1^{\alpha_1-2}),&a+2\leq \alpha_1\leq a+3,\\
(2a+2,\alpha_1),&\alpha_1=2a-1.
\end{array}\right.\]

For $(\alpha_2,\ldots,\alpha_h)=(a,a-1,3,1)$ with $a\geq 11$ let
\[\beta:=\left\{\begin{array}{ll}
(2a+3,5,1^{\alpha_1-5}),&a+5\leq \alpha_1\leq 2a-2\mbox{ or }2a\leq \alpha_1\leq 2a+3,\\
(2a+3,2,1^{\alpha_1-2}),&\alpha_1=a+1\mbox{ or }\alpha_1=a+4,\\
(2a+3,\alpha_1-2,1,1),&\alpha_1=a+2,\\
(2a+3,\alpha_1),&\alpha_1=a+3\mbox{ or }\alpha_1=2a-1.
\end{array}\right.\]

It's easy to check that in each of the above cases $\beta$ is a partition and that $h^\beta_{2,1}=\alpha_1$. In each of the above cases in can also be proved that $\chi^\beta_\alpha\not\in\{0,\pm1\}$.

Assume that $(\alpha_2,\ldots,\alpha_h)=(a,a-1,1)$ and $a+2\leq\alpha_1\leq 2a-2$, that $(\alpha_2,\ldots,\alpha_h)=(a,a-1,2,1)$ and $a+4\leq\alpha_1\leq 2a-2$ or that $(\alpha_2,\ldots,\alpha_h)=(a,a-1,3,1)$ and $a+5\leq\alpha_1\leq 2a-2$. In either case $h_{1,\beta_2+1}=2a-2\geq\alpha_1$. As $h_{2,1}^\beta=\alpha_1$ it follows from the Murnaghan-Nakayama formula that
\[\chi^\beta_\alpha=(-1)^{\alpha_1-\beta_2}\chi^{(|\alpha|-\alpha_1)}_{(\alpha_2,\ldots,\alpha_h)}+\chi^{(|\alpha|-2\alpha_1,\beta_2,1^{\alpha_1-\beta_2})}_{(\alpha_2,\ldots,\alpha_h)}.\]
Since by assumption
\begin{eqnarray*}
h_{3,1}^{(|\alpha|-2\alpha_1,\beta_2,1^{\alpha_1-\beta_2})}&=&\alpha_1-\beta_2\geq a,\\
h_{1,2}^{(|\alpha|-2\alpha_1,\beta_2,1^{\alpha_1-\beta_2})}&=&|\alpha|-2\alpha_1\leq a-2,
\end{eqnarray*}
and $\alpha_2=a$, we have that
\[\chi^\beta_\alpha=(-1)^{\alpha_1-\beta_2}+(-1)^{\alpha_2-1}\chi^{(|\alpha|-2\alpha_1,\beta_2,1^{\alpha_1-\beta_2-\alpha_2})}_{(\alpha_3,\ldots,\alpha_h)}.\]
By definition of $\beta$
\begin{eqnarray*}
h_{1,1}^{(|\alpha|-2\alpha_1,\beta_2,1^{\alpha_1-\beta_2-a})}&=&|\alpha|-2\alpha_1+\alpha_1-\beta_2-\alpha_2+1\\
&=&\alpha_3+\ldots+\alpha_h-(\alpha_4+\ldots+\alpha_h+1)+1\\
&=&\alpha_3.
\end{eqnarray*}
So
\[\chi^\beta_\alpha=(-1)^{\alpha_1-\beta_2}+(-1)^{\alpha_2-1+\alpha_1-\beta_2-\alpha_2+1}\chi^{(\beta_2-1)}_{(\alpha_4,\ldots,\alpha_h)}=(-1)^{\alpha_1-\beta_2}2.\]

The other cases can be computed similarly.

\section{Proof of Theorem \ref{t2} for $\alpha_2>\alpha_3+\ldots+\alpha_h$}\label{s2}




In this section we will prove Theorem \ref{t2} for $\alpha_2>\alpha_3+\ldots+\alpha_h$. Again, from Lemma \ref{l2}, as $\alpha\not\in\Sign$ but $(\alpha_2,\ldots,\alpha_h)\in\Sign$, we have that $\alpha_1\leq\alpha_2+\ldots+\alpha_h$.

Throughout this section let $k$ be minimal such that
\[\alpha_k+\ldots+\alpha_h<\alpha_1-\alpha_2.\]
Since $\alpha_1\leq \alpha_2+\ldots+\alpha_h$, it follows that $4\leq k\leq h+1$. Also define
\[x:=\alpha_k+\ldots+\alpha_h.\]


\begin{theor}
Assume that the following hold:
\begin{itemize}
\item
$\alpha\not\in\Sign$, $(\alpha_2,\ldots,\alpha_h)\in\Sign$ and $\alpha_1>\alpha_2>\alpha_3+\ldots+\alpha_h$,

\item
$k\leq h$,

\item
$\alpha_1-\alpha_2$ is not a part of $\alpha$,

\item
$\alpha_{k-1}>x$.
\end{itemize}
Then $\beta=(|\alpha|-\alpha_1,x+1,1^{\alpha_1-x-1})$ is a partition, $h_{2,1}^\beta=\alpha_1$ and $\chi^\beta_\alpha=(-1)^{\alpha_1-x-1}2$.
\end{theor}

\begin{proof}
By definition and by assumption
\[|\alpha|-\alpha_1=\alpha_2+\ldots+\alpha_h\geq \alpha_1\geq x+1,\]
from which follows that $\beta$ is a partition. Also clearly $h_{2,1}^\beta=\alpha_1$. We will now prove that $\chi^\beta_\alpha=(-1)^{\alpha_1-x-1}2$.

Assume first that $2\alpha_1+x>|\alpha|$. Then
\[2=|\alpha|-\alpha_1-(\alpha_2+\ldots+\alpha_h)+2\leq |\alpha|-2\alpha_1+2\leq x+1\]and so
\[h_{1,|\alpha|-2\alpha_1+2}^\beta=|\alpha|-\alpha_1+2-(|\alpha|-2\alpha_1+2)=\alpha_1.\]
It follows that
\[\chi^\beta_\alpha=(-1)^{\alpha_1-x-1}\chi^{(|\alpha|-\alpha_1)}_{(\alpha_2,\ldots,\alpha_h)}-\chi^\delta_{(\alpha_2,\ldots,\alpha_h)}=(-1)^{\alpha_1-x-1}-\chi^\delta_{(\alpha_2,\ldots,\alpha_h)},\]
where $\delta:=(x,|\alpha|-2\alpha_1+1,1^{\alpha_1-x-1})$. So it is enough to prove that $\chi^\delta_{(\alpha_2,\ldots,\alpha_h)}=(-1)^{\alpha_1-x}$. As $h_{1,2}^\delta\leq x<\alpha_{k-1}<\alpha_2$  by assumption, we have that
\[\chi^\delta_{(\alpha_2,\ldots,\alpha_h)}=(-1)^{\alpha_2-1}\chi^\epsilon_{(\alpha_3,\ldots,\alpha_h)},\]
where $\epsilon:=(x,|\alpha|-2\alpha_1+1,1^{\alpha_1-\alpha_2-x-1})$ (as by definition of $x$, $\alpha_1-\alpha_2>x$, so that $\epsilon$ is a partition). By minimality of $k$,
\[|\epsilon|<2x+\alpha_1-\alpha_2-x\leq 2x+\alpha_{k-1}.\]
Also, as $(\alpha_2,\ldots,\alpha_h)\in\Sign$ and $k-2\geq 2$,
\[\alpha_3+\ldots+\alpha_h=|\epsilon|<2(\alpha_k+\ldots+\alpha_h)+\alpha_{k-1}<\alpha_{k-2}+\ldots+\alpha_h\]
and then $k-2<3$. Since $k\geq 4$ it follows that $k=4$. As by induction $\alpha_3>x$,
\[\chi^\delta_{(\alpha_2,\ldots,\alpha_h)}=(-1)^{\alpha_2-1}\chi^\epsilon_{(\alpha_3,\ldots,\alpha_h)}=(-1)^{\alpha_2-1+\alpha_1-\alpha_2-x-1}\chi^{(x)}_{(\alpha_4,\ldots,\alpha_h)}=(-1)^{\alpha_1-x}\]
and then the theorem holds in this case.

Assume now that $2\alpha_1+x<|\alpha|$. Then
\[x+1<|\alpha|-2\alpha_1+1\leq |\alpha|-\alpha_1\]
and so
\[h_{1,|\alpha|-2\alpha_1+1}^\beta=|\alpha|-\alpha_1+1-(|\alpha|-2\alpha_1+1)=\alpha_1.\]
By definition $\alpha_2\leq \alpha_1-x-1$ and by assumption $\alpha_2>\alpha_3+\ldots+\alpha_h$, so that any partition of $\alpha_2+\ldots+\alpha_h$ has at most one hook of length $\alpha_2$. So
\begin{eqnarray*}
\chi^\beta_\alpha&=&(-1)^{\alpha_1-x-1}\chi^{(|\alpha|-\alpha_1)}_{(\alpha_2,\ldots,\alpha_h)}+\chi^{(|\alpha|-2\alpha_1,x+1,1^{\alpha_1-x-1})}_{(\alpha_2,\ldots,\alpha_h)}\\
&=&(-1)^{\alpha_1-x-1}+(-1)^{\alpha_2-1}\chi^\lambda_{(\alpha_3,\ldots,\alpha_h)},
\end{eqnarray*}
where $\lambda=(|\alpha|-2\alpha_1,x+1,1^{\alpha_1-\alpha_2-x-1})$. So it is enough to prove that $\chi^\lambda_{(\alpha_3,\ldots,\alpha_h)}=(-1)^{\alpha_1-\alpha_2-x}$.

First assume that $\alpha_{k-1}>\alpha_1-\alpha_2$. Then
\[h_{2,1}^\lambda=\alpha_1-\alpha_2<\alpha_j\]
for $3\leq j\leq k-1$ and
\[h_{1,x+2}^\lambda=|\lambda|-x-1-\alpha_1+\alpha_2\geq |\lambda|-\alpha_{k-1}-\ldots-\alpha_h=\alpha_3+\ldots+\alpha_{k-2}\]
if $x+2\leq \lambda_1$. If $\lambda_1=x+1$ then
\[|\lambda|=x+\alpha_1-\alpha_2+1\leq \alpha_{k-1}+\ldots+\alpha_h\leq \alpha_3+\ldots+\alpha_h=|\lambda|\]
and so in this case $k=4$. In either case
\begin{eqnarray*}
\chi^\lambda_{(\alpha_3,\ldots,\alpha_h)}&=&\chi^{(\alpha_{k-1}-\alpha_1+\alpha_2+x,x+1,1^{\alpha_1-\alpha_2-x-1})}_{(\alpha_{k-1},\ldots,\alpha_h)}\\
&=&(-1)^{\alpha_1-\alpha_2-x}\chi^{(x)}_{(\alpha_k,\ldots,\alpha_h)}\\
&=&(-1)^{\alpha_1-\alpha_2-x}
\end{eqnarray*}
and so the theorem holds also in this case.

Now assume that $\alpha_{k-1}<\alpha_1-\alpha_2$. Then $k\geq 5$ (otherwise $\alpha_1>\alpha_2+\ldots+\alpha_h$) and
\[\alpha_{k-1}+x=\alpha_{k-1}+\ldots+\alpha_h\geq \alpha_1-\alpha_2\]
by definition of $k$. Since $\alpha_1-\alpha_2-x-1<\alpha_{k-1}$ by minimality of $k$ and since by assumption $x<\alpha_{k-1}$ and $\alpha_1-\alpha_2$ is not a part of $\alpha$, it follows similarly to the previous case that
\[\chi^\lambda_{(\alpha_3,\ldots,\alpha_h)}=\chi^\mu_{(\alpha_{k-2},\ldots,\alpha_h)},\]
where $\mu:=(\alpha_{k-2}+\alpha_{k-1}-\alpha_1+\alpha_2+x,x+1,1^{\alpha_1-\alpha_2-x-1})$. As
\[2\leq\alpha_{k-1}-\alpha_1+\alpha_2+x+2\leq x+1\]
and so
\[h_{1,\alpha_{k-1}-\alpha_1+\alpha_2+x+2}^\mu\!=\!\alpha_{k-2}+\alpha_{k-1}-\alpha_1+\alpha_2+x+2-(\alpha_{k-1}-\alpha_1+\alpha_2+x+2)\!=\!\alpha_{k-2}.\]
From $\alpha_1-\alpha_2$ not being a part of $\alpha$ and
\[x,\alpha_1-\alpha_2-x-1<\alpha_{k-1}<\alpha_{k-2}\]
it follows that
\[\chi^\mu_{(\alpha_{k-2},\ldots,\alpha_h)}=-\chi^{\nu}_{(\alpha_{k-1},\ldots,\alpha_h)}=(-1)^{\alpha_1-\alpha_2-x},\]
with $\nu=(x,\alpha_{k-1}-\alpha_1+\alpha_2+x+1,1^{\alpha_1-\alpha_2-x-1})$, and so the theorem holds also in this case.

At last assume that $2\alpha_1+x=|\alpha|$. Then
\[\alpha_1=|\alpha|-\alpha_1-x=\alpha_2+\ldots+\alpha_{k-1}.\]
By definition of $k$ we then have that
\[\alpha_3+\ldots+\alpha_{k-1}=\alpha_1-\alpha_2\leq\alpha_{k-1}+\ldots+\alpha_h\]
and so
\[\alpha_3+\ldots+\alpha_{k-2}\leq\alpha_k+\ldots+\alpha_h.\]
If $k\geq 5$ then $k-2\geq 3$ and then $\alpha_{k-2}\leq \alpha_k+\ldots+\alpha_h$. This gives a contradiction with $(\alpha_2,\ldots,\alpha_h)\in\Sign$. So $k=4$ and then $\alpha_1-\alpha_2=\alpha_3$ is a part of $\alpha$, which contradicts the assumptions.
\end{proof}

\begin{theor}
Assume that the following hold:
\begin{itemize}
\item
$\alpha\not\in\Sign$, $(\alpha_2,\ldots,\alpha_h)\in\Sign$ and $\alpha_1>\alpha_2>\alpha_3+\ldots+\alpha_h$,

\item
$k\leq h$,

\item
$\alpha_1-\alpha_2$ is not a part of $\alpha$,

\item
$\alpha_{k-1}\leq x$,

\item
none of the following holds:
\begin{itemize}
\item
$(\alpha_{k-1},\ldots,\alpha_h)=(3,2,1,1)$ and $\alpha_1=\alpha_2+\alpha_{k-1}+\ldots+\alpha_h$,

\item
$(\alpha_{k-1},\ldots,\alpha_h)=(5,3,2,1)$ and $\alpha_1=\alpha_2+\alpha_{k-1}+\ldots+\alpha_h$,

\item
$(\alpha_{k-1},\ldots,\alpha_h)=(a,a-1,1)$ with $a\geq 2$ and $\alpha_1=\alpha_2+\alpha_{k-1}+\ldots+\alpha_h-1$,

\item
$(\alpha_{k-1},\ldots,\alpha_h)=(a,a-1,2,1)$ with $a\geq 4$ and $\alpha_1=\alpha_2+\alpha_{k-1}+\ldots+\alpha_h-3$,

\item
$(\alpha_{k-1},\ldots,\alpha_h)=(a,a-1,3,1)$ with $a\geq 5$ and $\alpha_1=\alpha_2+\alpha_{k-1}+\ldots+\alpha_h-4$.
\end{itemize}
\end{itemize}
Then $\beta=(|\alpha|-\alpha_1,x+1,1^{\alpha_1-x-1})$ is a partition, $h_{2,1}^\beta=\alpha_1$ and $\chi^\beta_\alpha=(-1)^{\alpha_1-x-1}2$.
\end{theor}

\begin{proof}
As in the previous theorem we have that $2\alpha_1+x\not=|\alpha|$, since $\alpha_1-\alpha_2$ is not a part of $\alpha$.

Assume first that $2\alpha_1+x>|\alpha|$. From the proof of the previous theorem ($\alpha_2>x$ since $(\alpha_2,\ldots,\alpha_h)\in\Sign$), it is enough to prove that $\chi^\epsilon_{(\alpha_3,\ldots,\alpha_h)}=(-1)^{\alpha_1-\alpha_2-x-1}$, where $\epsilon=(x,|\alpha|-2\alpha_1+1,1^{\alpha_1-\alpha_2-x-1})$. In this case it holds $k=4$ as in the previous theorem.

Assume now that $2\alpha_1+x<|\alpha|$. Since $\alpha_{k-1}\leq x<\alpha_1-\alpha_2$ we have that $\alpha_{k-1}<\alpha_1-\alpha_2$. As $\alpha_1-\alpha_2$ is not a part of $\alpha$ it is enough, from the proof of the previous theorem, to prove that $x<\alpha_j$ for $j\leq k-2$ and that $\chi^{\nu}_{(\alpha_{k-1},\ldots,\alpha_h)}=(-1)^{\alpha_1-\alpha_2-x-1}$, where $\nu=(x,\alpha_{k-1}-\alpha_1+\alpha_2+x+1,1^{\alpha_1-\alpha_2-x-1})$. In order to prove that $x<\alpha_j$ for $j\leq k-2$, it is enough to prove it for $j=k-2$. As $k\geq 4$, so that $k-2\geq 2$, and $(\alpha_2,\ldots,\alpha_h)\in\Sign$, we have that $x=\alpha_k+\ldots+\alpha_h<\alpha_{k-2}$.

In either case it is then enough to prove that $\chi^{\lambda_y}_{(\alpha_{k-1},\ldots,\alpha_h)}=(-1)^y$ for $\lambda_y=(x,\alpha_{k-1}-y,1^y)$, $y=\alpha_1-\alpha_2-x-1$. Notice that $0\leq y\leq\alpha_{k-1}-1$, since $\lambda_y$ is a partition.

Clearly $h_{2,1}^{\lambda_y}=\alpha_{k-1}$. If this is the only $\alpha_{k-1}$-hook of $\lambda$, then it is easy to see that $\chi^{\lambda_y}_{(\alpha_{k-1},\ldots,\alpha_h)}=(-1)^y$. Else, due to hooks lengths being decreasing along both the rows and the columns, $\lambda_y$ has exactly 2 $\alpha_{k-1}$-hooks and there exists $2\leq j\leq x$ with $h_{1,j}^{\lambda_y}=\alpha_{k-1}$.

As $\alpha_{k-1}\leq x$ by assumption
\begin{eqnarray*}
(\alpha_{k-1},\ldots,\alpha_h)&\in&\{(1,1),(3,2,1,1),(5,3,2,1)\}\cup\{(a,a-1,1):a\geq 2\}\\
&&\hspace{12pt}\cup\{(a,a-1,2,1):a\geq 4\}\cup\{(a,a-1,3,1):a\geq 5\}.
\end{eqnarray*}

If $(\alpha_{k-1},\ldots,\alpha_h)=(1,1)$ then $x=1<2$, so no such $j$ exists.

If $(\alpha_{k-1},\ldots,\alpha_h)=(3,2,1,1)$ then $\lambda_y\in\{(4,3),(4,1,1,1)\}$ if such a $j$ exists, and so $y=0$ or $y=2$ respectively. The second case would imply $\alpha_1-\alpha_2-x=3$, which would contradict the assumption. As $\chi^{(4,3)}_{(3,2,1,1)}=1=(-1)^0$ the theorem holds in this case. 

If $(\alpha_{k-1},\ldots,\alpha_h)=(5,3,2,1)$ and there exists such a $j$ then
\[\lambda_y\in\{(6,5),(6,4,1),(6,3,1,1),(6,1^5)\}\]
and then $y=0$, $y=1$, $y=2$ or $y=4$ respectively. In the last case $\alpha_1-\alpha_2-x=5$, which contradicts the assumption. In the other cases $\chi^{(6,5)}_{(5,3,2,1)}=1=(-1)^0$, $\chi^{(6,4,1)}_{(5,3,2,1)}=-1=(-1)^1$ and $\chi^{(6,3,1,1)}_{(5,3,2,1)}=1=(-1)^2$ and so the theorem holds also in this case.

If $(\alpha_{k-1},\ldots,\alpha_h)=(a,a-1,1)$ then there exists such a $j$ if and only if $0\leq y\leq \alpha_{k-1}-2$. If $y=\alpha_{k-1}-2$ then $\alpha_1-\alpha_2-x=\alpha_{k-1}-1$ which contradicts the assumption. In the other cases
\[\chi^{\lambda_y}_{(\alpha_{k-1},\ldots,\alpha_h)}=\chi^{(a,a-y,1^y)}_{(a,a-1,1)}=(-1)^y\chi^{(a)}_{(a-1,1)}-\chi^{(a-y-1,1^{y+1})}_{(a-1,1)}=(-1)^y,\]
since $a-y-2,y+1\geq 1$, so that also $a-y-2,y+1<a-1$. In particular the theorem holds in this case.

If $(\alpha_{k-1},\ldots,\alpha_h)=(a,a-1,2,1)$ then there exists such a $j$ if and only if $y\not=\alpha_{k-1}-3$. For $y=\alpha_{k-1}-4$ we have that $\alpha_1-\alpha_2-x=\alpha_{k-1}-3$, which contradicts the assumptions.

For $0\leq y\leq \alpha_{k-1}-5$ then $j=4$ as $\alpha_{k-1}-y>4$, so that
\[h_{1,4}^{\lambda_y}=a+2+2-4=a.\]
So
\[\chi^{\lambda_y}_{(\alpha_{k-1},\ldots,\alpha_h)}\!=\!\chi^{(a+2,a-y,1^y)}_{(a,a-1,2,1)}\!=\!(-1)^y\chi^{(a+2)}_{(a-1,2,1)}-\chi^{(a-y-1,3,1^y)}_{(a-1,2,1)}\!=\!(-1)^y-\chi^{(a-y-1,3,1^y)}_{(a-1,2,1)}\]
and
\[\chi^{(a-y-1,3,1^y)}_{(a-1,2,1)}=\left\{\begin{array}{ll}
0&y\not=0\\
-\chi^{(2,1)}_{(2,1)}=0&y=0,\\
\end{array}\right.\]
as
\begin{eqnarray*}
h_{1,1}^{(a-y-1,3,1^y)}&=&a,\\
h_{2,1}^{(a-y-1,3,1^y)}&=&y+3<a-1,\\
h_{1,2}^{(a-y-1,3,1^y)}&=&a-y-1\leq a-1,
\end{eqnarray*}
since $0\leq y\leq\alpha_{k-1}-5=a-5$. In particular $\chi^{\lambda_y}_{(\alpha_{k-1},\ldots,\alpha_h)}=(-1)^y$.

For $\alpha_{k-1}-2\leq y\leq\alpha_{k-1}-1$ then $j=3$ as $\alpha_{k-1}-y\leq 2$, so that
\[h_{1,3}^{\lambda_y}=a+2+1-3=a.\]
It follows that
\[\chi^{\lambda_y}_{(\alpha_{k-1},\ldots,\alpha_h)}=\chi^{(a+2,a-y,1^y)}_{(a,a-1,2,1)}=(-1)^y\chi^{(a+2)}_{(a-1,2,1)}+\chi^{(2,a-y,1^y)}_{(a-1,2,1)}=(-1)^y+\chi^{(2,a-y,1^y)}_{(a-1,2,1)}.\]
As
\[\chi^{(2,a-y,1^y)}_{(a-1,2,1)}=\left\{\begin{array}{ll}
\chi^{(2,2,1^{a-2})}_{(a-1,2,1)}=0&y=\alpha_{k-1}-2,\\
\chi^{(2,1^a)}_{(a-1,2,1)}=(-1)^{a-2}\chi^{(2,1)}_{(2,1)}=0&y=\alpha_{k-1}-1,
\end{array}\right.\]
as $a\geq 4$. In particular also in this case $\chi^{\lambda_y}_{(\alpha_{k-1},\ldots,\alpha_h)}=(-1)^y$.

If $(\alpha_{k-1},\ldots,\alpha_h)=(a,a-1,3,1)$ then there exists such a $j$ if and only if $y\not=\alpha_{k-1}-4$. If $y=\alpha_{k-1}-5$ then $\alpha_1-\alpha_2-x=\alpha_k-4$, in contradiction to the assumption.

For $0\leq y\leq \alpha_{k-1}-6$ then $j=5$ as $\alpha_{k-1}-y>5$ and then
\[h_{1,5}^{\lambda_y}=a+3+2-5=a.\]
So
\[\chi^{\lambda_y}_{(\alpha_{k-1},\ldots,\alpha_h)}\!=\!\chi^{(a+3,a-y,1^y)}_{(a,a-1,3,1)}\!=\!(-1)^y\chi^{(a+3)}_{(a-1,3,1)}-\chi^{(a-y-1,4,1^y)}_{(a-1,3,1)}\!=\!(-1)^y-\chi^{(a-y-1,3,1^y)}_{(a-1,3,1)}\]
and
\[\chi^{(a-y-1,4,1^y)}_{(a-1,3,1)}=\left\{\begin{array}{ll}
0&y\not=0,\\
-\chi^{(3,1)}_{(3,1)}=0&y=0,
\end{array}\right.\]
as
\begin{eqnarray*}
h_{1,1}^{(a-y-1,4,1^y)}&=&a,\\
h_{2,1}^{(a-y-1,4,1^y)}&=&y+4<a-1,\\
h_{1,2}^{(a-y-1,4,1^y)}&=&a-y-1\leq a-1,
\end{eqnarray*}
since $0\leq y\leq\alpha_{k-1}-6=a-6$. In particular $\chi^{\lambda_y}_{(\alpha_{k-1},\ldots,\alpha_h)}=(-1)^y$.

For $\alpha_{k-1}-3\leq y\leq\alpha_{k-1}-1$ then $j=4$ as $\alpha_{k-1}-y\leq 3$, so that
\[h_{1,4}^{\lambda_y}=a+3+1-4=a.\]
Then
\[\chi^{\lambda_y}_{(\alpha_{k-1},\ldots,\alpha_h)}=\chi^{(a+3,a-y,1^y)}_{(a,a-1,3,1)}=(-1)^y\chi^{(a+3)}_{(a-1,3,1)}+\chi^{(3,a-y,1^y)}_{(a-1,3,1)}=(-1)^y+\chi^{(3,a-y,1^y)}_{(a-1,3,1)}.\]
As
\[\chi^{(3,a-y,1^y)}_{(a-1,3,1)}=\left\{\begin{array}{ll}
\chi^{(3,3,1^{a-3})}_{(a-1,3,1)}=0&y=\alpha_{k-1}-3,\\
\chi^{(3,2,1^{a-2})}_{(a-1,3,1)}=0&y=\alpha_{k-1}-2,\\
\chi^{(3,1^a)}_{(a-1,3,1)}=(-1)^{a-2}\chi^{(3,1)}_{(3,1)}=0&y=\alpha_{k-1}-1,
\end{array}\right.\]
since $a\geq 5$ it follows that also in this case $\chi^{\lambda_y}_{(\alpha_{k-1},\ldots,\alpha_h)}=(-1)^y$.
\end{proof}

\begin{theor}
Assume that the following hold:
\begin{itemize}
\item
$\alpha\not\in\Sign$, $(\alpha_2,\ldots,\alpha_h)\in\Sign$ and $\alpha_1>\alpha_2>\alpha_3+\ldots+\alpha_h$,

\item
$k\leq h$,

\item
$\alpha_1-\alpha_2$ is not a part of $\alpha$,

\item
$(\alpha_{k-1},\ldots,\alpha_h)\in\{(3,2,1,1),(5,3,2,1)\}$,

\item
$\alpha_1=\alpha_2+\alpha_{k-1}+\ldots+\alpha_h$.
\end{itemize}
Let $c$ equal to 3 if $(\alpha_{k-1},\ldots,\alpha_h)=(3,2,1,1)$ or equal to 6 if $(\alpha_{k-1},\ldots,\alpha_h)=(5,3,2,1)$.

Then $\beta:=(|\alpha|-\alpha_1,\alpha_1-c,1^c)$ is a partition with $h_{2,1}^\beta=\alpha_1$ and $\chi^\beta_\alpha=(-1)^c2$.
\end{theor}

\begin{proof}
Since $c<\alpha_2<\alpha_1<\alpha_2+\ldots+\alpha_h=|\alpha|-\alpha_1$ by assumption on $\alpha$, it follows that $\beta$ is a partition. Clearly $h_{2,1}^\beta=\alpha_1$.

Also, from
\[2\leq \alpha_3+\ldots+\alpha_{k-2}+2<\alpha_3+\ldots+\alpha_h-c<\alpha_1-c\]
we have that
\begin{eqnarray*}
h_{1,\alpha_3+\ldots+\alpha_{k-2}+2}^\beta&=&|\alpha|-\alpha_1+2-(\alpha_3+\ldots+\alpha_{k-2}+2)\\
&=&\alpha_2+\ldots+\alpha_h-\alpha_3-\ldots-\alpha_{k-2}\\
&=&\alpha_2+\alpha_{k-1}+\ldots+\alpha_h\\
&=&\alpha_1.
\end{eqnarray*}

If $(\alpha_{k-1},\ldots,\alpha_h)=(3,2,1,1)$ let $d=3$. If instead $(\alpha_{k-1},\ldots,\alpha_h)=(5,3,2,1)$ let $d=4$. Notice that $c+d=\alpha_{k-1}+\ldots+\alpha_h-1$. Then by assumption
\[\alpha_1-c=\alpha_2+\alpha_{k-1}+\ldots+\alpha_h-c=\alpha_2+d+1.\]
It follows that
\[\chi^\beta_\alpha=(-1)^c\chi^{(|\alpha|-\alpha_1)}_{(\alpha_2,\ldots,\alpha_h)}-\chi^\delta_{(\alpha_2,\ldots,\alpha_h)}=(-1)^c-\chi^\delta_{(\alpha_2,\ldots,\alpha_h)}\]
where $\delta=(\alpha_2+d,\alpha_3+\ldots+\alpha_{k-2}+1,1^c)$.

Assume first that $k=4$. Then $\alpha_3+\ldots+\alpha_{k-2}=0$ and so, as $c+1<\alpha_2$,
\[\chi^\delta_{(\alpha_2,\ldots,\alpha_h)}=\chi^{(d,1^{c+1})}_{(\alpha_{k-1},\ldots,\alpha_h)}=(-1)^{c-1}\]
(the last equality follows from $(\alpha_{k-1},\ldots,\alpha_h)\in\{(3,2,1,1),(5,3,2,1)\}$ and from the definition of $c$ and $d$) and so in this case $\chi^\beta_\alpha=(-1)^c2$.

So assume now that $k>4$. As $(\alpha_2,\ldots,\alpha_h)\in\Sign$, it follows that $\alpha_j>\alpha_{k-1}+\ldots+\alpha_h$ for $j\leq k-2$. Also
\[\delta_2=\alpha_3+\ldots+\alpha_{k-2}+1\geq\alpha_3+1>d+2>2.\]
So
\[h_{1,d+2}^\delta=\alpha_2+d+2-(d+2)=\alpha_2\]
and then as by assumption $|\delta|=\alpha_2+\ldots+\alpha_h<2\alpha_2$, so that $\delta$ cannot have more than 1 hook of length $\alpha_2$,
\[\chi^\delta_{(\alpha_2,\ldots,\alpha_h)}=-\chi^\epsilon_{(\alpha_3,\ldots,\alpha_h)}\]
with $\epsilon=(\alpha_3+\ldots+\alpha_{k-2},d+1,1^c)$. As $h_{2,1}^\epsilon=c+d+1=\alpha_{k-1}+\ldots+\alpha_h<\alpha_j$ for $j\leq k-2$ and then in particular also $\alpha_{k-2}\geq d+1>2$, we have that
\[\chi^\epsilon_{(\alpha_3,\ldots,\alpha_h)}=\chi^{(\alpha_{k-2},d+1,1^c)}_{(\alpha_{k-2},\ldots,\alpha_h)}=-\chi^{(d,1^{c+1})}_{(\alpha_{k-1},\ldots,\alpha_h)}=(-1)^c.\]
In particular also in this case $\chi^\beta_\alpha=(-1)^c2$.
\end{proof}

\begin{theor}
Assume that the following hold:
\begin{itemize}
\item
$\alpha\not\in\Sign$, $(\alpha_2,\ldots,\alpha_h)\in\Sign$ and $\alpha_1>\alpha_2>\alpha_3+\ldots+\alpha_h$,

\item
$k\leq h$,

\item
$\alpha_1-\alpha_2$ is not a part of $\alpha$,

\item
one of the following holds:
\begin{itemize}
\item
$(\alpha_{k-1},\ldots,\alpha_h)=(a,a-1,1)$ with $a\geq 2$, $\alpha_1=\alpha_2+\alpha_{k-1}+\ldots+\alpha_h-1$ and $(\alpha_{k-2},\ldots,\alpha_h)\not\in\{(3,2,1,1),(5,3,2,1)\}$,

\item
$(\alpha_{k-1},\ldots,\alpha_h)=(a,a-1,2,1)$ with $a\geq 4$ and $\alpha_1=\alpha_2+\alpha_{k-1}+\ldots+\alpha_h-3$,

\item
$(\alpha_{k-1},\ldots,\alpha_h)=(a,a-1,3,1)$ with $a\geq 5$ and $\alpha_1=\alpha_2+\alpha_{k-1}+\ldots+\alpha_h-4$.
\end{itemize}
\end{itemize}
Then $\beta:=(|\alpha|-\alpha_1,1^{\alpha_1})$ is a partition with $h_{2,1}^\beta=\alpha_1$ and $\chi^\beta_\alpha=(-1)^{\alpha_1-1}2$.
\end{theor}

\begin{proof}
From the definition we clearly have that $\beta$ is a partition with $h_{2,1}^\beta=\alpha_1$.

Notice that from the assumptions $\alpha_1=\alpha_2+2a-1$. Also
\[|\alpha|-\alpha_1=\alpha_2+\ldots+\alpha_h>\alpha_2+2a-1=\alpha_1\]
and so, as $\alpha_2>\alpha_3+\ldots+\alpha_h$, so that any partition of $\alpha_2+\ldots+\alpha_h$ has at most one hook of length $\alpha_2$,
\begin{eqnarray*}
\chi^\beta_\alpha&=&(-1)^{\alpha_1-1}\chi^{(|\alpha|-\alpha_1)}_{(\alpha_2,\ldots,\alpha_h)}+\chi^{(|\alpha|-2\alpha_1,1^{\alpha_1})}_{(\alpha_2,\ldots,\alpha_h)}\\
&=&(-1)^{\alpha_1-1}+(-1)^{\alpha_2-1}\chi^{(|\alpha|-2\alpha_1,1^{\alpha_1-\alpha_2})}_{(\alpha_3,\ldots,\alpha_h)}\\
&=&(-1)^{\alpha_1-1}+(-1)^{\alpha_1}\chi^{(|\alpha|-2\alpha_1,1^{2a-1})}_{(\alpha_3,\ldots,\alpha_h)}.
\end{eqnarray*}

Assume first that either $k=4$ or $k>4$ and $\alpha_{k-2}\geq 2a$. Then, as $\alpha_{k-1}+\ldots+\alpha_h\geq 2a$ it follows that
\[\chi^{(|\alpha|-2\alpha_1,1^{2a-1})}_{(\alpha_3,\ldots,\alpha_h)}=\chi^{(\alpha_{k-1}+\ldots+\alpha_h-2a+1,1^{2a-1})}_{(\alpha_{k-1},\ldots,\alpha_h)}=(-1)^{(a-1)+(a-2)}=-1.\]
The second last equality follows from
\[(\alpha_{k-1}+\ldots+\alpha_h-2a+1,1^{2a-1})\!=\!\!\left\{\begin{array}{ll}
\!\!(1^{2a})&(\alpha_{k-1},\ldots,\alpha_h)=(a,a-1,1),\\
\!\!(3,1^{2a-1})&(\alpha_{k-1},\ldots,\alpha_h)=(a,a-1,2,1),\\
\!\!(4,1^{2a-1})&(\alpha_{k-1},\ldots,\alpha_h)=(a,a-1,3,1),
\end{array}\right.\]
so that, by assumption on $a$, $a-1>h_{1,2}^{(\alpha_{k-1}+\ldots+\alpha_h-2a+1,1^{2a-1})}$ in the last two cases.

Assume now that $k>4$ and $\alpha_{k-2}<2a\leq \alpha_{k-1}+\ldots+\alpha_h$. Notice that in this case $(\alpha_{k-1},\ldots,\alpha_h)=(a,a-1,1)$, as $(\alpha_2,\ldots,\alpha_h)\in\Sign$ and then also $(\alpha_{k-2},\ldots,\alpha_h)\in\Sign$. From this assumption and the assumption that $(\alpha_{k-2},\ldots,\alpha_h)\not\in\{(3,2,1,1),(5,3,2,1)\}$ it follows that $(\alpha_{k-2},\ldots,\alpha_h)\in\{(4,3,2,1),(5,4,3,1)\}$. Also, always by assumption of $(\alpha_2,\ldots,\alpha_h)\in\Sign$, if $k\geq 6$ then $\alpha_{k-3}>2a-1$. In either of the two cases
\[\chi^{(|\alpha|-2\alpha_1,1^{2a-1})}_{(\alpha_3,\ldots,\alpha_h)}=\chi^{(\alpha_{k-2}+1,1^{2a-1})}_{(\alpha_{k-2},\ldots,\alpha_h)}=-1.\]

In either case $\chi^\beta_\alpha=(-1)^{\alpha_1-1}2$ and so the theorem is proved.
\end{proof}

\begin{theor}
Assume that the following hold:
\begin{itemize}
\item
$\alpha\not\in\Sign$, $(\alpha_2,\ldots,\alpha_h)\in\Sign$ and $\alpha_1>\alpha_2>\alpha_3+\ldots+\alpha_h$,

\item
$k\leq h$,

\item
$\alpha_1-\alpha_2$ is not a part of $\alpha$,

\item
$\alpha_1=\alpha_2+\alpha_{k-1}+\ldots+\alpha_h-1$,

\item
$(\alpha_{k-2},\ldots,\alpha_h)\in\{(3,2,1,1),(5,3,2,1)\}$.
\end{itemize}
Then $\beta:=(|\alpha|-\alpha_1,\alpha_1)$ is a partition with $h_{2,1}^\beta=\alpha_1$ and $\chi^\beta_\alpha=2$.
\end{theor}

\begin{proof}
Since, by assumption, $\alpha_1<\alpha_2+\ldots+\alpha_h=|\alpha|-\alpha_1$ we have that $\beta$ is a partition. Also clearly $h_{2,1}^\beta=\alpha_1$.

Notice that in this case $k-2>2$, as $\alpha_{k-2}<\alpha_{k-1}+\ldots+\alpha_h$ and by assumption $\alpha_2>\alpha_3+\ldots+\alpha_h$. As
\[1<\alpha_3+\ldots+\alpha_{k-2}+3<\alpha_3+\ldots+\alpha_h<\alpha_2<\alpha_1\]
it follows that
\begin{eqnarray*}
h_{1,\alpha_3+\ldots+\alpha_{k-2}+3}^\beta&\!=\!&|\alpha|-\alpha_1+2-(\alpha_3+\ldots+\alpha_{k-2}+3)\\
&\!=\!&|\alpha|-(\alpha_2+\alpha_{k-1}+\ldots+\alpha_h-1)-(\alpha_3+\ldots+\alpha_{k-2})-1\\
&\!=\!&|\alpha|-\alpha_2-\ldots-\alpha_h\\
&\!=\!&\alpha_1.
\end{eqnarray*}
So
\[\chi^\beta_\alpha=\chi^{(|\alpha|-\alpha_1)}_{(\alpha_2,\ldots,\alpha_h)}-\chi^\delta_{(\alpha_2,\ldots,\alpha_h)}=1-\chi^\delta_{(\alpha_2,\ldots,\alpha_h)},\]
with
\[\delta:=(\alpha_1-1,\alpha_3+\ldots+\alpha_{k-2}+2)=(\alpha_2+\alpha_{k-1}+\ldots+\alpha_h-2,\alpha_3+\ldots+\alpha_{k-2}+2).\]
Also by assumption
\[1<\alpha_{k-1}+\ldots+\alpha_h<\alpha_{k-2}+2\leq \alpha_3+\ldots+\alpha_{k-2}+2\]
and then
\[h_{1,\alpha_{k-1}+\ldots+\alpha_h}^\delta=\alpha_2+\alpha_{k-1}+\ldots+\alpha_h-2+2-\alpha_{k-1}+\ldots+\alpha_h=\alpha_2.\]
From the previous $\alpha_3+\ldots+\alpha_{k-2}+2<\alpha_2$ and so
\[\chi^\delta_{(\alpha_2,\ldots,\alpha_h)}=-\chi^\epsilon_{(\alpha_3,\ldots,\alpha_h)}\]
with
\[\epsilon:=(\alpha_3+\ldots+\alpha_{k-2}+1,\alpha_{k-1}+\ldots+\alpha_h-1).\]
As $(\alpha_2,\ldots,\alpha_h)\in\Sign$ by assumption, so that $\alpha_j>\alpha_{k-1}+\ldots+\alpha_h>\epsilon_2$ for $j\leq k-3$ and as $\alpha_{k-2}+1>\alpha_{k-1}+\ldots+\alpha_h-1$ by assumption, it follows that
\[\chi^\epsilon_{(\alpha_3,\ldots,\alpha_h)}=\chi^{(\alpha_{k-2}+1,\alpha_{k-1}+\ldots+\alpha_h-1)}_{(\alpha_{k-2},\ldots,\alpha_h)}=1\]
(the last equation follows from the assumption that $(\alpha_{k-2},\ldots,\alpha_h)$ is either $(3,2,1,1)$ or $(5,3,2,1)$).

In particular $\chi^\beta_\alpha=2$ and so the theorem holds.
\end{proof}

\begin{theor}
Assume that the following hold:
\begin{itemize}
\item
$\alpha\not\in\Sign$, $(\alpha_2,\ldots,\alpha_h)\in\Sign$ and $\alpha_1>\alpha_2>\alpha_3+\ldots+\alpha_h$,

\item
$k\leq h$,

\item
there exists $i$ with $\alpha_i=\alpha_1-\alpha_2$,

\item
$\alpha_i\geq \alpha_{i+1}+\ldots+\alpha_h$.
\end{itemize}
Then $\beta=(|\alpha|-\alpha_1,\alpha_2+1,1^{\alpha_1-\alpha_2-1})$ is a partition with $h_{2,1}^\beta=\alpha_1$ and $\chi^\beta_\alpha=(-1)^{\alpha_1-\alpha_2-1}2$.
\end{theor}

\begin{proof}
Since by assumption $\alpha_1>\alpha_2+\alpha_h\geq \alpha_2+1$ and (also using Lemma \ref{l2})
\[|\alpha|-\alpha_1\geq\alpha_1>\alpha_2+\alpha_h\geq \alpha_2+1\]
it follows that $\beta$ is partition. Also clearly $h_{2,1}^\beta=\alpha_1$.

From the definition of $k$ and from
\[2\alpha_2>\alpha_2+\ldots+\alpha_h\geq \alpha_1\]
we have that $3\leq i<k\leq h$. Then
\begin{eqnarray*}
h_{1,2}^\beta&=&|\alpha|-\alpha_1=\alpha_2+\ldots+\alpha_h\geq \alpha_2+\alpha_i+\alpha_h>\alpha_1,\\
h_{1,\alpha_2+1}^\beta&=&|\alpha|-\alpha_1+2-\alpha_2-1=\alpha_3+\ldots+\alpha_h+1\leq \alpha_2<\alpha_1.
\end{eqnarray*}
In particular there exists $3\leq j\leq \alpha_2$ such that $h_{1,j}^\beta=\alpha_1$. From the Murnaghan-Nakayama formula it follows that
\begin{eqnarray*}
\chi^\beta_\alpha&=&(-1)^{\alpha_1-\alpha_2-1}\chi^{(|\alpha|-\alpha_1)}_{(\alpha_2,\ldots,\alpha_h)}-\chi^{(\alpha_2,j-1,1^{\alpha_1-\alpha_2-1})}_{(\alpha_2,\ldots,\alpha_3)}\\
&=&(-1)^{\alpha_1-\alpha_2-1}+\chi^{(j-2,1^{\alpha_1-\alpha_2})}_{(\alpha_3,\ldots,\alpha_h)}\\
&=&(-1)^{\alpha_1-\alpha_2-1}+\chi^{(\alpha_{i+1}+\ldots+\alpha_h,1^{\alpha_i})}_{(\alpha_i,\ldots,\alpha_h)}\\
&=&(-1)^{\alpha_1-\alpha_2-1}+(-1)^{\alpha_i-1}\chi^{(\alpha_{i+1}+\ldots+\alpha_h)}_{(\alpha_{i+1},\ldots,\alpha_h)}\\
&=&(-1)^{\alpha_1-\alpha_2-1}2.
\end{eqnarray*}
The second line follows from $h_{1,2}^{(\alpha_2,j-1,1^{\alpha_1-\alpha_2-1})}=\alpha_2$, as $j\geq 3$, and from $|(\alpha_2,j-1,1^{\alpha_1-\alpha_2-1})|=|\alpha|-\alpha_1<2\alpha_2$, so that $(\alpha_2,j-1,1^{\alpha_1-\alpha_2-1})$ has at most one hook of length $\alpha_2$. The third line from $\alpha_j>\alpha_i$ for $j<i$ and from $i<h$, so that
\begin{eqnarray*}
h_{1,2}^{(j-2,1^{\alpha_1-\alpha_2})}&=&|(\alpha_3,\ldots,\alpha_h)|-(\alpha_1-\alpha_2)-1\\
&=&\alpha_3+\ldots+\alpha_h-\alpha_i-1\\
&\geq&\alpha_1+\ldots+\alpha_{i+1}.
\end{eqnarray*}
The fourth line follows from $\alpha_i\geq \alpha_{i+1}+\ldots+\alpha_h$.
\end{proof}

\begin{theor}
Assume that the following hold:
\begin{itemize}
\item
$\alpha\not\in\Sign$, $(\alpha_2,\ldots,\alpha_h)\in\Sign$ and $\alpha_1>\alpha_2>\alpha_3+\ldots+\alpha_h$,

\item
$k\leq h$,

\item
there exists $i$ with $\alpha_i=\alpha_1-\alpha_2$,

\item
$\alpha_i<\alpha_{i+1}+\ldots+\alpha_h$.
\end{itemize}
Then $\beta=(|\alpha|-\alpha_1,\alpha_2+2,1^{\alpha_1-\alpha_2-2})$ is a partition with $h_{2,1}^\beta=\alpha_1$ and $\chi^\beta_\alpha=(-1)^{\alpha_1-\alpha_2}2$.
\end{theor}

\begin{proof}
Since by assumption $\alpha_1>\alpha_2+\alpha_h\geq \alpha_2+1$ and
\[|\alpha|-\alpha_1\geq\alpha_1>\alpha_2+\alpha_h\geq \alpha_2+1\]
it follows that $\beta$ is partition with $h_{2,1}^\beta=\alpha_1$.

From $\alpha_i<\alpha_{i+1}+\ldots+\alpha_h$ and $(\alpha_2,\ldots,\alpha_h)\in\Sign$ it follows that
\begin{eqnarray*}
(\alpha_i,\ldots,\alpha_h)&\in&\{(3,2,1,1),(5,3,2,1)\}\cup\{(a,a-1,2,1):a\geq 4\}\\
&&\hspace{12pt}\cup\{(a,a-1,3,1):a\geq 5\}.
\end{eqnarray*}

Similar to the previous theorem we have that $3\leq i<k\leq h$, from which follows that
\begin{eqnarray*}
h_{1,2}^\beta&=&|\alpha|-\alpha_1=\alpha_2+\ldots+\alpha_h\geq \alpha_2+\alpha_i+\ldots+\alpha_h\geq\alpha_1+2,\\
h_{1,\alpha_2+2}^\beta&=&|\alpha|-\alpha_1+2-\alpha_2-2=\alpha_3+\ldots+\alpha_h< \alpha_2<\alpha_1.
\end{eqnarray*}
In particular there exists $4\leq j\leq \alpha_2$ such that $h_{1,j}^\beta=\alpha_1$. So
\begin{eqnarray*}
\chi^\beta_\alpha&=&(-1)^{\alpha_1-\alpha_2-2}\chi^{(|\alpha|-\alpha_1)}_{(\alpha_2,\ldots,\alpha_h)}-\chi^{(\alpha_2+1,j-1,1^{\alpha_1-\alpha_2-2})}_{(\alpha_2,\ldots,\alpha_3)}\\
&=&(-1)^{\alpha_1-\alpha_2}+\chi^{(j-2,2,1^{\alpha_1-\alpha_2-2})}_{(\alpha_3,\ldots,\alpha_h)}\\
&=&(-1)^{\alpha_1-\alpha_2}+\chi^{(\alpha_{i+1}+\ldots+\alpha_h,2,1^{\alpha_i-2})}_{(\alpha_i,\ldots,\alpha_h)}.
\end{eqnarray*}
The second line follows from $\alpha_2>\alpha_3+\ldots+\alpha_h$ and, as $j\geq 4$,
\[h_{1,3}^{(\alpha_2+1,j-1,1^{\alpha_1-\alpha_2-2})}=\alpha_2+1+2-3=\alpha_2.\]
The third line follows from $\alpha_j>\alpha_i$ for $j<i$ and from
\begin{eqnarray*}
h_{1,3}^{(j-2,2,1^{\alpha_1-\alpha_2-2})}&=&|(\alpha_3,\ldots,\alpha_h)|-(\alpha_1-\alpha_2)-2\\
&=&\alpha_3+\ldots+\alpha_h-\alpha_i-2\\
&\geq&\alpha_1+\ldots+\alpha_{i+1}.
\end{eqnarray*}

If $(\alpha_i,\ldots,\alpha_h)\in\{(3,2,1,1),(5,3,2,1)\}$ it is easy to check that 
\[\chi^{(\alpha_{i+1}+\ldots+\alpha_h,2,1^{\alpha_i-2})}_{(\alpha_i,\ldots,\alpha_h)}=-1=(-1)^{\alpha_i}=(-1)^{\alpha_1-\alpha_2}.\]
In particular the theorem holds in this case.

If $(\alpha_i,\ldots,\alpha_h)=(a,a-1,c,1)$ with $c\in\{2,3\}$ then, as $a-1>c$,
\begin{eqnarray*}
\chi^{(\alpha_{i+1}+\ldots+\alpha_h,2,1^{\alpha_i-2})}_{(\alpha_i,\ldots,\alpha_h)}&=&\chi^{(a+c,2,1^{a-2})}_{(a,a-1,c,1)}\\
&=&(-1)^{a-2}\chi^{(a+c)}_{(a-1,c,1)}+\chi^{(c,2,1^{a-2})}\\
&=&(-1)^a\\
&=&(-1)^{\alpha_1-\alpha_2},
\end{eqnarray*}
so that the theorem holds also in this case.
\end{proof}

In the next theorems we will consider the case $k=h+1$, that is $\alpha_1-\alpha_2\leq \alpha_h$.

\begin{theor}
Assume that the following hold:
\begin{itemize}
\item
$\alpha\not\in\Sign$, $(\alpha_2,\ldots,\alpha_h)\in\Sign$ and $\alpha_1>\alpha_2>\alpha_3+\ldots+\alpha_h$,

\item
$\alpha_1-\alpha_2<\alpha_h$.
\end{itemize}
Then $\beta:=(|\alpha|-\alpha_1,1^{\alpha_1})$ is a partition with $h_{2,1}^\beta=\alpha_1$ and $\chi^\beta_\alpha=(-1)^{\alpha_1-1}2$.
\end{theor}

\begin{proof}
Clearly $\beta$ is a partition and $h_{2,1}^\beta=\alpha_1$. By assumption $|\alpha|-\alpha_1\geq\alpha_2+\alpha_h>\alpha_1$, from which also follows that $\alpha_1-\alpha_2<\alpha_h\leq\alpha_j$ for $j\leq h$. Also as by assumption $\alpha_2>\alpha_3+\ldots+\alpha_h$, so that any partition of $\alpha_2+\ldots+\alpha_h$ has at most one $\alpha_2$-hook, it follows from the Murnaghan-Nakayama formula that
\begin{eqnarray*}
\chi^\beta_\alpha&=&(-1)^{\alpha_1-1}\chi^{(|\alpha|-\alpha_1)}_{(\alpha_2,\ldots,\alpha_h)}+\chi^{(|\alpha|-2\alpha_1,1^{\alpha_1})}_{(\alpha_2,\ldots,\alpha_h)}\\
&=&(-1)^{\alpha_1-1}+(-1)^{\alpha_2-1}\chi^{(|\alpha|-2\alpha_1,1^{\alpha_1-\alpha_2})}_{(\alpha_3,\ldots,\alpha_h)}\\
&=&(-1)^{\alpha_1-1}+(-1)^{\alpha_2-1}\chi^{(\alpha_h-\alpha_1+\alpha_2,1^{\alpha_1-\alpha_2})}_{(\alpha_h)}\\
&=&(-1)^{\alpha_1-1}2.
\end{eqnarray*}
\end{proof}

\begin{theor}
Assume that the following hold:
\begin{itemize}
\item
$\alpha\not\in\Sign$, $(\alpha_2,\ldots,\alpha_h)\in\Sign$ and $\alpha_1>\alpha_2>\alpha_3+\ldots+\alpha_h$,

\item
$\alpha_1-\alpha_2=\alpha_h$,

\item
$h=3$.
\end{itemize}
Then $\beta=(\alpha_1,\alpha_1)$ is a partitions with $h_{2,1}^\beta=\alpha_1$ and $\chi^\beta_\alpha=2$.
\end{theor}

\begin{proof}
Notice that $\alpha_3\geq 2$, since $1\leq\alpha_1-\alpha_2=\alpha_3$ and $(\alpha_1,\alpha_2,\alpha_3)\not\in\Sign$. Clearly $\beta$ is a partition with $h_{2,1}^\beta=\alpha_1$.

As $\beta=(\alpha_1,\alpha_1)$ and $\alpha_3\geq 2$ we have that
\[\chi^\beta_\alpha=\chi^{(\alpha_1)}_{(\alpha_2,\alpha_3)}-\chi^{(\alpha_1-1,1)}_{(\alpha_2,\alpha_3)}=2.\]
\end{proof}

\begin{theor}
Assume that the following hold:
\begin{itemize}
\item
$\alpha\not\in\Sign$, $(\alpha_2,\ldots,\alpha_h)\in\Sign$ and $\alpha_1>\alpha_2>\alpha_3+\ldots+\alpha_h$,

\item
$\alpha_1-\alpha_2=\alpha_h\geq 2$,

\item
$h\geq 4$.
\end{itemize}
Then $\beta=(|\alpha|-\alpha_1,\alpha_2+2,1^{\alpha_1-\alpha_2-2})$ is a partition with $h_{2,1}^\beta=\alpha_1$ and $\chi^\beta_\alpha=(-1)^{\alpha_1-\alpha_2}2$.
\end{theor}

\begin{proof}
As $\alpha_2+2\leq\alpha_2+\alpha_h=\alpha_1$ and $|\alpha|-\alpha_1\geq\alpha_2+\alpha_h$ we have that $\beta$ is a partition and that $h_{2,1}^\beta=\alpha_1$. Notice that $\beta_1'$, which is the number of parts of $\beta$, is given by
\[\beta_1'=\alpha_1-\alpha_2=\alpha_h.\]
As $h\geq 4$ and $\alpha_{h-1}>\alpha_h\geq 2$ we have that
\begin{eqnarray*}
h_{1,2}^\beta&=&|\alpha|-\alpha_1\geq\alpha_2+\alpha_h+\alpha_{h-1}\geq \alpha_1+3,\\
h_{1,\alpha_2+2}^\beta&=&|\alpha|-\alpha_1-\alpha_2=\alpha_3+\ldots+\alpha_h\leq\alpha_2-1\leq\alpha_1-2.
\end{eqnarray*}
In particular there exists $5\leq j\leq \alpha_2$ with $h_{1,j}^\beta=\alpha_1$. Such $j$ satisfies $\beta\setminus R_{1,j}^\beta=(\alpha_2+1,j-1,1^{\alpha_1-\alpha_2-2})$ and then also $h_{1,3}^{\beta\setminus R_{1,j}^\beta}=\alpha_2$ as $j-1>3$ (where $R_{1,j}^\beta$ is the rim hook of $\beta$ corresponding to node $(1,j)$). As $\alpha_2>\alpha_3+\ldots+\alpha_h$, as $\beta_1'=\alpha_h$ and as $\alpha_i>\alpha_h$ for $i<h$ (since $\alpha_h\geq 2$) we then obtain from the Murnaghan-Nakayama formula that
\begin{eqnarray*}
\chi^\beta_\alpha&=&(-1)^{\alpha_1-\alpha_2-2}\chi^{(|\alpha|-\alpha_1)}_{(\alpha_2,\ldots,\alpha_h)}-\chi^{(\alpha_2+1,j-1,1^{\alpha_1-\alpha_2-2})}_{(\alpha_2,\ldots,\alpha_h)}\\
&=&(-1)^{\alpha_1-\alpha_2}+\chi^{(j-2,2,1^{\alpha_h-2})}_{(\alpha_3,\ldots,\alpha_h)}\\
&=&(-1)^{\alpha_1-\alpha_2}+\chi^{(\alpha_{h-1},2,1^{\alpha_h-2})}_{(\alpha_{h-1},\alpha_h)}\\
&=&(-1)^{\alpha_1-\alpha_2}-\chi^{(1^{\alpha_h})}_{(\alpha_{h-1},\alpha_h)}\\
&=&(-1)^{\alpha_1-\alpha_2}+(-1)^{\alpha_h}\\
&=&(-1)^{\alpha_1-\alpha_2}2.
\end{eqnarray*}
\end{proof}

\begin{theor}
Assume that the following hold:
\begin{itemize}
\item
$\alpha\not\in\Sign$, $(\alpha_2,\ldots,\alpha_h)\in\Sign$ and $\alpha_1>\alpha_2>\alpha_3+\ldots+\alpha_h$,

\item
$\alpha_1-\alpha_2=\alpha_h=1=\alpha_{h-1}$,

\item
$h\geq 4$.
\end{itemize}
Then $\beta=(|\alpha|-\alpha_1,\alpha_1)$ is a partition with $h_{2,1}^\beta=\alpha_1$ and $\chi^\beta_\alpha=2$.
\end{theor}

\begin{proof}
From Lemma \ref{l2} it follows from the assumptions that $|\alpha|-\alpha_1\geq \alpha_1$ and so $\beta$ is a partition. Also $h_{2,1}^\beta=\alpha_1$. As
\[3=\alpha_{h-1}+2\leq\alpha_3+\ldots+\alpha_{h-1}+2=\alpha_3+\ldots+\alpha_h+1\leq\alpha_2<\alpha_1\]
and
\[|\alpha|-2\alpha_1+2=\alpha_2+\ldots+\alpha_h-\alpha_1+2=\alpha_3+\ldots+\alpha_{h-1}+2,\]
we have that, for $j=|\alpha|-2\alpha_1+2$,
\[h_{1,j}^\beta=|\alpha|-\alpha_1+2-j=\alpha_1.\]
Also $2\leq j-1<\alpha_2$ and then, as $\alpha_2=\alpha_1-1$ and $\alpha_{h-2}>\alpha_{h-1}=\alpha_h=1$,
\[\chi^\beta_\alpha=\chi^{(|\alpha|-\alpha_1)}_{(\alpha_2\ldots,\alpha_h)}-\chi^{(\alpha_1-1,j-1)}_{(\alpha_2,\ldots,\alpha_h)}=1+\chi^{(j-2,1)}_{(\alpha_3,\ldots,\alpha_h)}=2.\]
\end{proof}

\begin{theor}
Assume that the following hold:
\begin{itemize}
\item
$\alpha\not\in\Sign$, $(\alpha_2,\ldots,\alpha_h)\in\Sign$ and $\alpha_1>\alpha_2>\alpha_3+\ldots+\alpha_h$,

\item
$\alpha_1-\alpha_2=\alpha_h=1<\alpha_{h-1}$,

\item
$h=4$.
\end{itemize}
Then $\beta=(\alpha_1-2,\alpha_3,\alpha_3,4,1^{\alpha_1-\alpha_3-2})$ is a partition with $h_{2,1}^\beta=\alpha_1$ and $\chi^\beta_\alpha=(-1)^{\alpha_1-\alpha_3}2$.
\end{theor}

\begin{proof}
Notice that from the assumptions it follows that $\alpha_3\geq 4$. Also $\alpha_1>\alpha_2>\alpha_3$ and so $\beta$ is a partition with $h_{2,1}^\beta=\alpha_1$. As $\alpha_2=\alpha_1-1$ and $\alpha_4=1$ we have that
\begin{eqnarray*}
\chi^\beta_\alpha&=&(-1)^{\alpha_1-\alpha_3}\chi^{(\alpha_1-2,\alpha_3-1,3)}_{(\alpha_1-1,\alpha_3,1)}-\chi^{(\alpha_3-1,\alpha_3-1,3,1^{\alpha_1-\alpha_3-1})}_{(\alpha_1-1,\alpha_3,1)}\\
&=&(-1)^{\alpha_1-\alpha_3}\chi^{(\alpha_3-2,2,1)}_{(\alpha_3,1)}+(-1)^{\alpha_1-\alpha_3+1}\chi^{(\alpha_3-1,2)}_{(\alpha_3,1)}\\
&=&(-1)^{\alpha_1-\alpha_3}2.
\end{eqnarray*}
\end{proof}

\begin{theor}
Assume that the following hold:
\begin{itemize}
\item
$\alpha\not\in\Sign$, $(\alpha_2,\ldots,\alpha_h)\in\Sign$ and $\alpha_1>\alpha_2>\alpha_3+\ldots+\alpha_h$,

\item
$\alpha_1-\alpha_2=\alpha_h=1$,

\item
$h\geq 5$,

\item
$\alpha_{h-1}=2$.
\end{itemize}
Then $\beta=(|\alpha|-\alpha_1-2,\alpha_1-2,2,2)$ is a partition with $h_{2,1}^\beta=\alpha_1$ and $\chi^\beta_\alpha=-2$.
\end{theor}

\begin{proof}
As $\alpha_1>\alpha_2>\ldots>\alpha_h=1$ it follows that $\alpha_1\geq h\geq 5$.
Also, by assumption on $\alpha$,
\[|\alpha|-\alpha_1\geq \alpha_2+\alpha_{h-2}+\alpha_h\geq \alpha_1+3\]
and so it follows that $\beta$ is a partition. Clearly $h_{2,1}^\beta=\alpha_1$. Since by assumption
\[|\alpha|-2\alpha_1+2=\alpha_2+\ldots+\alpha_h-\alpha_1+2=\alpha_3+\ldots+\alpha_h+1\leq\alpha_2<\alpha_1
\]
we also have that
\begin{eqnarray*}
h_{1,3}^\beta&=&|\alpha|-\alpha_1-2+2-3=|\alpha|-\alpha_1-3\geq\alpha_1,\\
h_{1,\alpha_1-2}^\beta&=&|\alpha|-\alpha_1-2+2-\alpha_1+2=|\alpha|-2\alpha_1+2<\alpha_1.
\end{eqnarray*}
In particular there exists $3\leq j\leq \alpha_1-3$ with $h_{1,j}^\beta=\alpha_1$. 

From $\alpha_{h-1}=2$ and $\alpha_h=1$ it follows that $\alpha_j+\ldots+\alpha_h-3\geq\alpha_j$ for $j\leq h-2$. Since $\alpha_j\geq 3$ for $j\leq h-2$ we then have that
\begin{eqnarray*}
\chi^\beta_\alpha&=&\chi^{(|\alpha|-\alpha_1-2,1,1)}_{(\alpha_2,\ldots,\alpha_h)}-\chi^{(\alpha_1-3,j-1,2,2)}_{(\alpha_2,\ldots,\alpha_h)}\\
&=&\chi^{(|\alpha|-\alpha_1-\alpha_2-2,1,1)}_{(\alpha_3,\ldots,\alpha_h)}+\chi^{(j-2,1,1)}_{(\alpha_3,\ldots,\alpha_h)}\\
&=&2\chi^{(1,1,1)}_{(2,1)}\\
&=&-2.
\end{eqnarray*}
\end{proof}

%
%

\begin{theor}
Assume that the following hold:
\begin{itemize}
\item
$\alpha\not\in\Sign$, $(\alpha_2,\ldots,\alpha_h)\in\Sign$ and $\alpha_1>\alpha_2>\alpha_3+\ldots+\alpha_h$,

\item
$\alpha_1-\alpha_2=\alpha_h=1$,

\item
$h\geq 5$,

\item
$\alpha_{h-1}\geq 3$.
\end{itemize}
Then $\beta=(|\alpha|-\alpha_1-\alpha_{h-1}+1,3,3,2^{\alpha_{h-1}-3},1^{\alpha_1-\alpha_{h-1}-1})$ is a partition with $h_{2,1}^\beta=\alpha_1$ and $\chi^\beta_\alpha=(-1)^{\alpha_1+\alpha_{h-1}-1}2$.
\end{theor}

\begin{proof}
As $h\geq 5$, so that
\[\beta_1=|\alpha|-\alpha_1-\alpha_{h-1}+1\geq \alpha_2+\alpha_3+1>\alpha_1+3,\]
and as $\alpha_1>\alpha_{h-1}\geq 3$ it follows that $\beta$ is a partition with $h_{2,1}^\beta=\alpha_1$. Also $\beta_1\geq 4$ and $h_{1,4}^\beta\geq\alpha_1$. From the assumptions we also have
\[|\alpha|-2\alpha_1-\alpha_{h-1}=\alpha_2+\ldots+\alpha_h-\alpha_1-\alpha_{h-1}=\alpha_3+\ldots+\alpha_{h-2}>\alpha_3+\ldots+\alpha_{h-3}+2.\]
Since $\alpha_j>\alpha_{h-1}$ for $j<h-1$ and again any partition of $\alpha_2+\ldots+\alpha_h$ has at most one $\alpha_2$-hook, we have that
\begin{eqnarray*}
\chi^\beta_\alpha&\!\!\!=\!\!\!&(-1)^{\alpha_1-3}\chi^{(|\alpha|-\alpha_1-\alpha_{h-1}+1,2,1^{\alpha_{h-1}-3})}_{(\alpha_2,\ldots,\alpha_h)}\!+\!\chi^{(|\alpha|-2\alpha_1-\alpha_{h-1}+1,3,3,2^{\alpha_{h-1}-3},1^{\alpha_1-\alpha_{h-1}-1})}_{(\alpha_2,\ldots,\alpha_h)}\\
&\!\!\!=\!\!\!&(-1)^{\alpha_1-1}\chi^{(|\alpha|-2\alpha_1-\alpha_{h-1}+2,2,1^{\alpha_{h-1}-3})}_{(\alpha_3,\ldots,\alpha_h)}\!+\!(-1)^{\alpha_1-4}\chi^{(|\alpha|-2\alpha_1-\alpha_{h-1}+1,3,1^{\alpha_{h-1}-3})}_{(\alpha_3,\ldots,\alpha_h)}\\
&\!\!\!=\!\!\!&(-1)^{\alpha_1-1}\chi^{(\alpha_{h-2}+2,2,1^{\alpha_{h-1}-3})}_{(\alpha_{h-2},\alpha_{h-1},\alpha_h)}\!+\!(-1)^{\alpha_1}\chi^{(\alpha_{h-2}+1,3,1^{\alpha_{h-1}-3})}_{(\alpha_{h-2},\alpha_{h-1},\alpha_h)}\\
&\!\!\!=\!\!\!&(-1)^{\alpha_1-1}2\chi^{(2,2,1^{\alpha_{h-1}-3})}_{(\alpha_{h-1},\alpha_h)}\\
&\!\!\!=\!\!\!&(-1)^{\alpha_1+\alpha_{h-1}-1}2.
\end{eqnarray*}
\end{proof}

\section{The partitions $(\gamma_{s+1},\ldots,\gamma_r)$ are sign partitions}\label{s3}

In this section we will prove that 
\begin{itemize}
\item
$()$, $(1,1)$, $(3,2,1,1)$, $(5,3,2,1)$,

\item
$(a,a-1,1)$ with $a\geq 2$,

\item
$(a,a-1,2,1)$ with $a\geq 4$,

\item
$(a,a-1,3,1)$ with $a\geq 5$
\end{itemize}
are all sign partitions. For $()$, $(1,1)$, $(3,2,1,1)$ and $(5,3,2,1)$ this can be done by just looking at the corresponding character table. For the other partitions we will use the next lemma.

\begin{lemma}\label{l4}
Let $a\geq 2$ and $\gamma=(a,a-1,\gamma_3,\ldots,\gamma_r)$ be a partition. Assume that the following hold.
\begin{itemize}
\item
$(a-1,\gamma_3,\ldots,\gamma_r)$ is a sign partition,

\item
$\gamma_3+\ldots+\gamma_r\leq a$.
\end{itemize}
If $\beta$ is a partition of $|\gamma|$ for which $\chi^\beta_\gamma\not\in\{0,\pm 1\}$ then $\beta$ has two $a$-hooks. Also if $\delta$ is obtained from $\beta$ by removing an $a$-hook then $\chi^\delta_{(a-1,\gamma_3,\ldots,\gamma_r)}\not=0$. In particular each such $\delta$ has an $(a-1)$-hook.
\end{lemma}

\begin{proof}
By assumption
\[|\gamma|=2a-1+\gamma_3+\ldots+\gamma_r<3a.\]
In particular any partition of $|\gamma|$ has at most two $a$-hooks. As
\[\chi^\beta_\gamma=\sum_{(i,j):h_{i,j}^\beta=a}\pm\chi^{\beta\setminus R_{i,j}^\beta}_{(a-1,\gamma_3,\ldots,\gamma_r)}\]
and, since $(a-1,\gamma_3,\ldots,\gamma_r)$ is a sign partition, so that $\chi^{\beta\setminus R_{i,j}^\beta}_{(a-1,\gamma_3,\ldots,\gamma_r)}\in\{0,\pm 1\}$ for each $(i,j)\in[\beta]$, the Young diagram of $\beta$, with $h_{i,j}^\beta=a$, the lemma follows.
\end{proof}

\begin{theor}
If $a\geq 2$ then $(a,a-1,1)$ is a sign partition.
\end{theor}

\begin{proof}
As $(a-1,1)$ is a sign partition for $a\geq 2$, from Lemma \ref{l4} we only need to check that $\chi^\beta_{(a,a-1,1)}\in\{0,\pm1\}$ for partitions $\beta$ of $2a$ with two $a$-hooks and such that if $\mu$ and $\nu$ are the partitions obtained from $\beta$ by removing an $a$-hook then $\mu$ and $\nu$ both have an an $(a-1)$-hook. From $\beta$ having two $a$-hooks it follows that $\mu$ and $\nu$ also have an $a$-hook.  The only partitions of $a$ having both an $a$-hook and an $(a-1)$-hook are $(a)$ and $(1^a)$. As $\mu\not=\nu$ it then follows that $\{\mu,\nu\}=\{(a),(1^a)\}$. Looking at the $a$-quotients and $a$-cores of $\beta$, $\mu$ and $\nu$ we have that there exists a unique such $\beta$, which is given by $\beta=(a,2,1^{a-2})$. We have
\[\chi^{(a,2,1^{a-2})}_{(a,a-1,1)}=(-1)^{a-2}\chi^{(a)}_{(a-1,1)}-\chi^{(1^a)}_{(a-1,1)}=(-1)^a+(-1)^{a-1}=0\]
and so $(a,a-1,1)$ is a sign partition.
\end{proof}

\begin{theor}
If $a\geq 4$ then $(a,a-1,2,1)$ is a sign partition.
\end{theor}

\begin{proof}
For $a=4$ we can check that $(a,a-1,2,1)=(4,3,2,1)$ is a sign partition by looking at the character table of $S_{10}$. So assume that $a\geq 5$. As $(a-1,2,1)$ is a sign partition for $a\geq 5$ from Lemma \ref{l2}, from Lemma \ref{l4} we only need to check that $\chi^\beta_{(a,a-1,2,1)}\in\{0,\pm1\}$ for partitions $\beta$ of $2a+2$ with two $a$-hooks and such that if $\mu$ and $\nu$ are the partitions obtained from $\beta$ by removing an $a$-hook then $\mu$ and $\nu$ have both an $a$-hook and an $(a-1)$-hook.

So let $\beta$ have two $a$-hook. Then, as $|\beta|=2a+2<3a$, we have that $\beta_{(a)}$, the $a$-core of $\beta$, is either $(2)$ or $(1^2)$. We will assume that $\beta_{(a)}=(2)$, since for any partitions $\lambda,\rho$ with $|\lambda|=|\rho|$ and any positive integer $q$, we have that $\chi^\lambda_\rho=\pm\chi^{\lambda'}_\rho$ and $\lambda'_{(q)}=(\lambda_{(q)})'$, where $\lambda'$ is the adjoint partition of $\lambda$ and similarly for $\lambda_{(q)}$. Then $\mu$ and $\nu$ can be obtained by adding an $a$-hook to $(2)$ and so
\[\mu,\nu\in\{(a+2),(2,2,1^{a-2}),(2,1^a)\}\cup\{(a-i,3,1^{i-1}):1\leq i\leq a-3\},\]
as all these partitions can be obtained by adding an $a$-hook to $(2)$ and, since $2<a$, there are exactly $a$ such partitions. As $\mu$ and $\nu$ have an $(a-1)$-hook we then have that
\[\mu,\nu\in\{(a+2),(2,1^a),(a-1,3),(3,3,1^{a-4})\}.\]
Notice that since $a\geq 5$ the four above partitions are distinct.  As $a\geq 5$
\begin{eqnarray*}
\chi^{(2,1^a)}_{(a-1,2,1)}&=&(-1)^{a-2}\chi^{(2,1)}_{(2,1)}=0,\\
\chi^{(a-1,3)}_{(a-1,2,1)}&=&-\chi^{(2,1)}_{(2,1)}=0,
\end{eqnarray*}
we only need to consider, from Lemma \ref{l4}, the partition $\beta$ corresponding to $\{\mu,\nu\}=\{(a+2),(3,3,1^{a-4})\}$, that is for $\beta=(a+2,4,1^{a-4})$. As
\[\chi^{(a+2,4,1^{a-4})}_{(a,a-1,2,1)}=-\chi^{(3,3,1^{a-4})}_{(a-1,2,1)}+(-1)^{a-4}\chi^{(a+2)}_{(a-1,2,1)}=(-1)^{a-3}\chi^{(3)}_{(2,1)}+(-1)^a=0\]
it follows that $(a,a-1,2,1)$ is a sign partition.
\end{proof}

\begin{theor}
If $a\geq 5$ then $(a,a-1,3,1)$ is a sign partition.
\end{theor}

\begin{proof}
If $a=5$ then $(a,a-1,3,1)=(5,4,3,1)$ and by looking at the character table of $S_{13}$ we can easily check that this is a sign partition. So assume now that $a\geq 6$. As $(a-1,3,1)$ is a sign partition for $a\geq 6$ from Lemma \ref{l2}, from Lemma \ref{l4} we only need to check that $\chi^\beta_{(a,a-1,3,1)}\in\{0,\pm1\}$ for partitions $\beta$ of $2a+3$ with two $a$-hooks and such that if $\mu$ and $\nu$ are the partitions obtained from $\beta$ by removing an $a$-hook then $\mu$ and $\nu$ have both an $a$-hook and an $(a-1)$-hook.

So let $\beta$ have two $a$-hook. Then $\beta_{(a)}$ is $(3)$, $(2,1)$ or $(1^3)$. Similarly to the previous theorem we will assume that $\beta_{(a)}$ is either $(3)$ or $(2,1)$.

Assume first that $\beta_{(a)}=(3)$. Then, as $\mu$ and $\nu$ can be obtained by adding an $a$-hook to $(3)$ and as there exists exactly $a$ such partitions since $a>3$,
\[\mu,\nu\!\in\!\{(a+3),(3,3,1^{a-3}),(3,2,1^{a-2}),(3,1^a)\}\cup\{(a-i,4,1^{i-1}):1\leq i\leq a-4\}.\]
As $\mu$ and $\nu$ also have an $(a-1)$-hook it then follows that
\[\mu,\nu\in\{(a+3),(3,1^a),(a-1,4),(4,4,1^{a-5})\}.\]
As $a\geq 6$
\begin{eqnarray*}
\chi^{(3,1^a)}_{(a-1,3,1)}&=&(-1)^{a-2}\chi^{(3,1)}_{(3,1)}=0,\\
\chi^{(a-1,4)}_{(a-1,3,1)}&=&-\chi^{(3,1)}=0
\end{eqnarray*}
and so, from Lemma \ref{l4}, we can assume that $\{\gamma,\delta\}=\{(a+3),(4,4,1^{a-5})\}$, that is that $\beta=(a+3,5,1^{a-5})$ and then
\[\chi^\beta_{(a,a-1,3,1)}=-\chi^{(4,4,1^{a-5})}_{(a-1,3,1)}+(-1)^{a-5}\chi^{(a+3)}_{(a-1,3,1)}=(-1)^{a-4}\chi^{(4)}_{(3,1)}+(-1)^{a-5}=0.\]

Assume now that $\beta_{(a)}=(2,1)$. Also in this case, as $a>3$, there exist exactly $a$ partitions which can be obtained by adding an $a$-hook to $(2,1)$ and $\mu$ and $\nu$ are two of them. So
\[\mu,\nu\!\in\!\!\{(a+2,1),(a,3),(2,2,2,1^{a-3}),(2,1^{a+1})\}\cup\{(a-i,3,2,1^{i-2})\!:\!2\!\leq\! i\!\leq\! a-3\}\hspace{-0.4pt}.\]
As $\mu$ and $\nu$ have an $(a-1)$-hook it follows that
\[\mu,\nu\in\{(a+2,1),(a,3),(2,2,2,1^{a-3}),(2,1^{a+1}),(a-2,3,2),(3,3,2,1^{a-5})\}.\]
Since $a\geq 6$
\begin{eqnarray*}
\chi^{(a+2,1)}_{(a-1,3,1)}&=&\chi^{(3,1)}_{(3,1)}=0,\\
\chi^{(2,1^{a+1})}_{(a-1,3,1)}&=&(-1)^{a-2}\chi^{(2,1,1)}_{(3,1)}=0,\\
\chi^{(a-2,3,2)}_{(a-1,3,1)}&=&\chi^{(2,1,1)}_{(3,1)}=0,\\
\chi^{(3,3,2,1^{a-5})}_{(a-1,3,1)}&=&(-1)^{a-4}\chi^{(3,1)}_{(3,1)}=0
\end{eqnarray*}
we again only need to consider one partition $\beta$. In this case $\{\mu,\nu\}=\{(a,3),(2,2,2,1^{a-3})\}$ and then $\beta=(a,3,3,1^{a-3})$. As
\[\chi^{(a,3,3,1^{a-3})}_{(a,a-1,3,1)}\!=\!\chi^{(2,2,2,1^{a-3})}_{(a-1,3,1)}+(-1)^{a-3}\chi^{(a,3)}_{(a-1,3,1)}\!=\!(-1)^{a-3}\chi^{(2,2)}_{(3,1)}+(-1)^{a-2}\chi^{(2,2)}_{(3,1)}\!=\!0,\]
it follows that $(a,a-1,3,1)$ is a sign partition also for $a\geq 6$.
\end{proof}

\section{Proof of Theorem \ref{t1}}

For $r\leq 2$ Theorem \ref{t1} follows from Lemmas \ref{l1} and \ref{l2}. So assume now that $r\geq 3$.

From Lemma \ref{l2} and Section \ref{s3} it easily follows that if $\gamma\in\Sign$ then $\gamma$ is a sign partition.

Assume now that $\gamma=(\gamma_1,\ldots,\gamma_r)$ is a sign partition. From Lemma \ref{l1} it follows that $(\gamma_{r-1},\gamma_r)\in\Sign$. Also from Lemma \ref{l2}, $\gamma_{i-1}>\gamma_i$ for $2\leq i\leq r-1$. Fix $2\leq i\leq r-1$ and assume that $(\gamma_i,\ldots,\gamma_r)\in\Sign$.

Assume that $(\gamma_{i-1},\ldots,\gamma_r)\not=(5,4,3,2,1)$ and that $(\gamma_{i-1},\ldots,\gamma_r)\not\in\Sign$. From Theorem \ref{t2} we can find $\beta$ such that $\chi^\beta_{(\gamma_{i-1},\ldots,\gamma_r)}\not\in\{0,\pm 1\}$ and $h_{2,1}^\beta=\gamma_{i-1}$. Let
\[\delta:=(\beta_1+\gamma_1+\ldots+\gamma_{i-2},\beta_2,\beta_3,\ldots).\]
Then $\delta$ is a partition of $|\gamma|$. If $i-1=1$ then
\[\chi^\delta_\gamma=\chi^\beta_{(\gamma_{i-1},\ldots,\gamma_r)}\not\in\{0,\pm 1\},\]
in contradiction to $\gamma$ being a sign partition. If $i-1\geq 2$ then $(1,\beta_1+1)\in[\delta]$ and
\[h_{1,\beta_1+1}^\delta=\gamma_1+\ldots+\gamma_{i-2}.\]
Since $\beta_2<\beta_1+1$ and $h_{2,1}^\delta=h_{2,1}^\beta=\gamma_{i-1}<\gamma_j$ for $j\leq i-2$, we have that also in this case
\[\chi^\delta_\gamma=\chi^\beta_{(\gamma_{i-1},\ldots,\gamma_r)}\not\in\{0,\pm 1\},\]
which again gives a contradiction.

Assume now that $(\gamma_{i-1},\ldots,\gamma_r)=(5,4,3,2,1)$. If $i-1=1$ or $i-1\geq 2$ and $\gamma_{i-2}\geq 7$, then similarly to the previous case
\[\chi^{(4+\gamma_1+\ldots+\gamma_{i-2},4,4,3)}_\gamma=\chi^{(4,4,4,3)}_{(5,4,3,2,1)}=-2.\]
If $i-1\geq 2$ and $\gamma_{i-1}=6$ we have similarly that
\[\chi^{(15+\gamma_1+\ldots+\gamma_{i-3},2,1,1,1,1)}_\gamma=\chi^{(15,2,1,1,1,1)}_{(6,5,4,3,2,1)}=2.\]
In either case we have a contradiction with $\gamma$ being a sign partition.

So $(\gamma_{i-1},\ldots,\gamma_r)\in\Sign$. By induction $\gamma\in\Sign$ and so Theorem \ref{t1} is proved.

\section*{Acknowledgements}

Part of the work contained in this paper is contained in the author's master thesis
(\cite{m1}), which was written at the University of Copenhagen, under the supervision of J\o rn B. Olsson, whom the author thanks for reviewing the paper.

\end{document}